\documentclass{amsart}
\usepackage{amssymb,euscript,tikz,units,mathrsfs}
\usepackage[colorlinks,citecolor=blue,linkcolor=red]{hyperref}
\usepackage{verbatim}
\usepackage[nameinlink]{cleveref}
\usepackage{colonequals}
\usepackage{tikz}
\usepackage[titletoc]{appendix}
\usepackage{yhmath}
\usepackage{enumerate}
\usepackage{hyperref}
\usepackage{url}
\usepackage{enumitem}
\usepackage{fancyhdr} 
\usepackage{lastpage}
\usepackage[all]{xy}
\usepackage{xcolor}
\theoremstyle{plain}
\newtheorem{theorem}{Theorem}[section]
\newtheorem{corollary}[theorem]{Corollary}
\newtheorem{proposition}[theorem]{Proposition}
\newtheorem{lemma}[theorem]{Lemma}

\newtheorem{question}[theorem]{Question}
\newtheorem{remark}[theorem]{Remark}

\newcommand{\be}{\begin{equation}}
\newcommand{\ene}{\end{equation}}
\newcommand{\br}{\begin{remark}}
\newcommand{\er}{\end{remark}}
\newcommand{\bl}{\begin{lem}}
\newcommand{\el}{\end{lem}}
\newcommand{\bcor}{\begin{cor}}
\newcommand{\ecor}{\end{cor}}
\newcommand{\bpro}{\begin{pro}}
\newcommand{\epro}{\end{pro}}
\newcommand{\ben}{\begin{enumerate}}
\newcommand{\een}{\end{enumerate}}
\newcommand{\bp}{\begin{proof}}
\newcommand{\ep}{\end{proof}}
\newcommand{\bpo}{\begin{pro}}
\newcommand{\epo}{\end{pro}}
\newcommand{\beq}{\begin{equation*}}
\newcommand{\eeq}{\end{equation*}}
\newcommand{\bear}{\begin{eqnarray}}
\newcommand{\eear}{\end{eqnarray}}
\newcommand{\beqar}{\begin{eqnarray*}}
\newcommand{\eeqar}{\end{eqnarray*}}
\newcommand{\bt}{\begin{theorem}}
\newcommand{\et}{\end{theorem}}
\newcommand{\bex}{\begin{excer}}
\newcommand{\eex}{\end{excer}}

\theoremstyle{definition}
\newtheorem{definition}[theorem]{Definition}

\theoremstyle{remark}

\newtheorem*{con*}{Construction}
\newtheorem*{rem*}{Remark}
\newtheorem*{exam*}{Example}
\newtheorem*{exams*}{Examples}
\newtheorem*{thm*}{\bf Theorem}
\newtheorem*{que*}{Question}

\newtheorem*{Def*}{Definition}
\newtheorem*{Cons*}{Construction}
\newtheorem*{Lem*}{Lemma}
\newtheorem*{Conj*}{\bf Conjecture}

\numberwithin{equation}{section} \numberwithin{figure}{section}

\begin{document}

\title{The Spectrum Zero Problem of nonlinear Dirac equation with particle-antiparticle interaction}
 
\author{ Qi Guo, Yuanyuan Ke*, Bernhard Ruf } 

\maketitle

\noindent{ \bf Abstract  }:
{\small 
In this study, we investigate the {\it Spectrum Zero Problem} of nonlinear Dirac equations with a focus on the behavior of zero at the boundaries of the spectral gap. We introduce a nonlinear particle-antiparticle interaction and demonstrate that the problem exhibits asymmetric behavior at the left and right boundaries of the spectrum. Specifically, when zero is at the right boundary, the problem has only trivial solutions and is identified as a bifurcation point on the left, whereas nontrivial solutions exist when zero is at the left boundary or within the spectral gap. The main idea is to employ a variational method involving a perturbation technique that places zero within the spectral gap. We use the critical point theorem of the perturbed functional to construct a Palais-Smale sequence in order to approach the critical point of the target energy functional. Additionally, we utilize the concentration-compactness principle to identify critical points of the original functional and explore the associated bifurcation phenomena. Our results reveal an asymmetric phenomenon in nonlinear quantum systems and provide insights into why strongly indefinite problems typically address zero only at the left boundary of the spectral gap.
 }

\noindent {\bf Keywords}:  {\small   Nonlinear Dirac equations, Spectrum zero problem, Solitary waves.}

\noindent {\bf AMS} Subject Classification: \small 35A15, 35M30, 35Q60, 47J10, 81Q10

 \medskip
\noindent{\small *Corresponding Author: keyy@ruc.edu.cn}
\tableofcontents

\section{Introduction and main results}
The {\it Spectrum Zero Problem} initially proposed by Ambrosetti \cite{DonDinGuo23} refers to the variational problems when zero is in the essential spectrum of a linear operator in the working space. This problem is divided into two cases: zero belonging to the interior of the essential spectrum and zero belonging to the boundary. This paper mainly focuses on the second case of the {\it Spectrum Zero Problem} which is also called the spectrum point zero problem in some references, i.e. reaction-diffusion equations \cite{ChenTang22, WeiYang14}, nonlinear Schr\"odinger equations \cite{BartschDing99, CAM19, Mederski15, Schechter15, WillemZou}, and Hamiltonian systems \cite{SunChenChu}. The linear operators in these problems have infinite-dimensional negative spaces which are known as the strongly indefinite cases. Most existing results focus on operators with spectra bounded below. However, standard methods like spectral projection approximation \cite{BartschDing99, SunChenChu, WeiYang14}, the Nehari–Pankov manifold method \cite{ChenTang22, Mederski15}, and the modified weak linking theorem \cite{Schechter15, WillemZou} are inapplicable here due to the spectrum of the Dirac operator being unbounded both above and below. Recently, the first author and collaborators \cite{DonDinGuo23} tackled this issue by constructing a new proper workspace and perturbing the functional. Previous studies, which dealt with zero at the left boundary of the spectrum, yielded solutions with low regularity ($H^1_{loc}$ for Dirac, $H^2_{loc}$  for Schr\"odinger). The Pohozaev identity indicates that  only trivial strong solutions ($H^1$ for Dirac, $H^2$ for Schr\"odinger) exist when zero is at the right boundary. This discrepancy highlights differences between zero at the left and right boundaries of the spectrum, leading to the following question:

\begin{question}\label{question0}
   In the {\it Spectrum Zero Problem}, is there an essential difference between the two cases that zero at the left  or the right boundaries of the spectrum?
\end{question}

This paper aims to answer this question. We will show that for nonlinear Dirac equations with particle-antiparticle interaction, nontrivial strong solutions exist when zero is at the left boundary of the  spectrum, while only trivial strong solutions exist when zero is at the right boundary.

\subsection{Physical Motivation}

In this part, we review the origin of the nonlinear terms in the Dirac equation, which have been studied for various purposes 
\cite{MR1897689, WOS:A1988P347300005,WOS:000267267900023,WOS:000276973700004}.
It is worth mentioning that in Schr\"odinger equations, the nonlinearities occur as an approximation in optics and condensed matter.
At the quantum mechanical level, it is hoped to detect quantum nonlinearities at high energy or at very short distances.
Indeed, neutrino oscillations may be related to quantum nonlinearities \cite{WOS:000276973700004}.
Nonlinear versions of Dirac equations have been used as effective theories in atomic, nuclear, and gravitational physics, see \cite{Finkelstein1}.
One of its general forms is
\begin{align*}\label{nonlinear}
i  c\hbar \gamma^\mu\partial_\mu\psi-mc^2  \psi- G_{\bar{\psi}}+\partial_\mu ( G_{\partial_\mu \bar{\psi}})=0,
\end{align*}
where $\psi:\mathbb{R}^3\rightarrow \mathbb{C}^4$ represents the wave function of the state of Dirac particles, such as electron, $\bar{\psi}$ is the Dirac adjoint of $\psi$, $m>0$ is the mass of the particle, $c$ is the speed of light,
$\hbar$ is the Planck constant, $G$ is a real-valued function of the wave function $\psi$, its adjoint and their derivatives, $G_{\bar{\psi}}$ is the differential of $G$ with respect to the column vector $\bar{\psi}$.
If $G$ satisfies
\[
\overline{G_\psi}-\partial_\mu (\overline{G_{\partial_\mu\psi}})=G_{\bar{\psi}}-\partial_\mu ( G_{\partial_\mu \bar{\psi}}),
\]
then the nonlinear Dirac equation can be obtained from the action with the Lagrangian density given by
\[\mathcal{L}=ic\hbar \bar{\psi} \gamma^\mu \partial_\mu \psi-mc^2\bar{\psi}\psi- G(\psi,\bar{\psi},\partial_\mu \psi,\partial_\mu\bar{\psi}).\]
 To construct quantum nonlinearities of Dirac equations, one needs to consider the basic properties of linear theory, such as locality, Poincar\'e invariance, hermiticity, current conservation, universality, separability, and discrete symmetries.
By demanding additional conditions on $G$, some of these properties can be preserved, see \cite{ WOS:000267267900023}.
 We study solitary wave solutions in this paper, which are solutions of the form
\[\psi(t,x)=e^{-i\omega t }u(x),\]
where $\omega$ is called the frequency of the wave function. Solitons are known to exist in many systems based on the Dirac equations, see \cite{BC}.
The frequency $\omega$ represents a parameter describing the time evolution of the spinor field $\psi$.
Note that in the linear theory, the frequency $\omega$ equals the energy of the field. In order to simplify the model, we set $c=\hbar=1$ and use a function $F:\mathbb{C}^4\rightarrow \mathbb{R}$ instead of $G$, where $F$ satisfies $$F_u(u):=\nabla F(u)=\gamma^0 \left(G_{\bar{u}}-\partial_\mu ( G_{\partial_\mu \bar{u}})\right)(u,\bar{u},\partial_\mu u,\partial_\mu \bar{u}).$$
Then $u(x)$ solves the following stationary nonlinear Dirac equation
\begin{align}\label{eq1.1}\tag{NDE}
-i\alpha\cdot \nabla u+m\beta u-F_u(u)=\omega u \ ,
\end{align}
where  $\alpha$ and $\beta$ denote the Dirac matrices, see section \ref{sec2.1} below.

\subsection{Mathematical Aspects of Nonlinear Models}
The Spectrum Zero Problem concerns when frequency $\omega$ allows solutions in nonlinear Dirac equations \eqref{eq1.1} and when it does not. 
There are many results for the nonlinear Dirac equations with different frequencies and assumptions on nonlinearities.
 For the case of the Soler-type nonlinearity $F( u)=\frac{1}{2}H(\bar{u} u)$, with $\bar{u}u=\langle\beta u,u\rangle$. Comech, Guan, Gustafson found one positive and one negative eigenvalue present in the spectrum of the linearizations at small amplitude solitary waves in the limit $\omega\rightarrow m$, which implies that these solitary waves are linearly unstable \cite{CGG}. For the case $\omega\in (m-\varepsilon, m)$ for some $\varepsilon$ small enough, Boussa\"id and Comech proved the absence of eigenvalues with positive real part, which implies the spectral stability of small amplitude soliton waves \cite{BC2}.
There are also many results on this problem in \cite{MR968485, MR2434346,  MR949625} with $\omega\in (0,m)$ by searching for solutions that are separable in spherical coordinates.
The authors use the shooting method to show existence and multiplicity results of $f$ and $g$.
It is worth mentioning that in \cite{MR456046}, V\'{a}zquez found there are no localized solutions of the Soler-type with $|\omega|>m$.
And in \cite{MR1071235}, Balabane, Cazenave and V\'{a}zquez showed the existence of stationary states with compact support for Dirac fields with singular nonlinearities and $\omega>m$.
In \cite{MR3070757}, Treust studied nodal solutions for Dirac equations with singular nonlinearities.
Variational methods were applied to search for these solutions by Esteban and S\'er\'e when $\omega\in (0,m)$, see \cite{ MR2434346,Esteban-Sere1995CMP, MR1897689}.
For the non-autonomous case, Ding and Ruf studied the case that a linear potential is of Coulomb-type or special scalar potentials, and nonlinearities have the form $R_u(x,u)$ in \cite{Ding-Ruf2012SIAM}.


Nonlinear phenomena are becoming increasingly more prevalent in various contexts, from the properties of different media to the groundbreaking research of Baker et al. in \cite{baker}.
As they demonstrated, the use of a laser beam can be used to slow down antihydrogen atoms, which can have a significant impact on understanding the fundamental symmetries.
To take advantage of this, nonlinear terms must be introduced to the model that can differentiate between particles and their antiparticles.
These nonlinear terms are useful for improving the accuracy, precision, and generalization ability of the model.
To do so, we must first recall the free Dirac operator, denoted by $D=-i \alpha\cdot \nabla+m \beta$, which is self-adjoint on $L^2(\mathbb{R}^3,\mathbb{C}^4)$ with domain $H^1(\mathbb{R}^3,\mathbb{C}^4)$ and formal domain $H^{1/2}(\mathbb{R}^3,\mathbb{C}^4)$.
Based on the spectrum of $D$, we have that $L^2(\mathbb{R}^3,\mathbb{C}^4)$ possesses the following orthogonal decomposition:
\[L^2=L^+\oplus L^-,\]
with $D$ being positive (or negative) definite on $L^+$ (or $L^-$).
Let $E$ be the completion of $\mathscr{D}(|D|^{1/2})$ under the following real inner product
\[ ( u,v) :=\Re (|D|^{1/2}u,|D|^{1/2}v)_{L^2},\]
where $(\cdot, \cdot)_{L^2}$ is the complex inner product in $L^2(\mathbb{R}^3,\mathbb{C}^4)$. The induced norm is denoted by $\|\cdot\|$. Then $E$ possesses the following decomposition
\[E=E^+\oplus E^-,\]
where $E^\pm=E\cap L^\pm$. The orthogonal projection operators onto the subspaces are given by
\[P^\pm=\frac{1}{2}\left(I\pm  |D|^{-1}D\right).\]
Here the absolute operator $|D|$ is given as $|D|=\sqrt{-\Delta+m^2} I$. We denote $u^\pm=P^\pm u$.
In the standard model, a state $u\in E$ is a superposition of particles and its antiparticles, $u^+\in E^+$ describes the state of Dirac fermions with positive energy, and $u^-\in E^-$ describes its antiparticles with negative energy, which can cancel part of the energy of $u^+$.

\subsection{Main Results}
Consider a class of nonlinear terms, denoted as $F$, that differentiate between Dirac particles and their antiparticles.
This allows us to better understand the behavior of these particles and to explore the underlying physics at work. Hence we introduce a function $\hat{F}:\mathbb{C}^4\oplus \mathbb{C}^4\rightarrow \mathbb{R}$ and set $F(u)=\hat{F}(u^+,u^-)$.
Next, we give the basic assumptions on $\hat{F}$, which means our results can only be applied to these determined nonlinear models. In the following, $p > 0$ is a given real number.

\begin{itemize}
\item[$(F_1)$]  $\hat{F}\in \mathcal{C}^1\left(\mathbb{C}^4\times \mathbb{C}^4,\mathbb{R}\right)$.
\item[$(F_2)$]  There is $a_1, a_2>0$, such that for any $s,t\in \mathbb{C}^4$, we have $$a_1 \left(|s|^p+|t|^2\right)\leq \hat{F}(s,t)\leq a_2 (|s|^p+|t|^p+|t|^2).$$
\item[$(F_3)$]  There is $b_1, b_2>0$, such that for any $s,t\in \mathbb{C}^4$, we have
\[2\hat{F}(s,t) +b_1 |s|^p \leq \langle \partial_s \hat{F} , s\rangle +\langle\partial_t \hat{F} , t\rangle\leq 3\hat{F}(s,t)-b_2\left(|s|^p+|t|^2\right).\]

\item[$(F_4)$] There is $ c_1>0$, $d_2\geq d_1>0$, such that for any $s,t\in \mathbb{C}^4$, we have
\[\langle \partial_s \hat{F} ,s\rangle\leq c_1 |s|^p+d_1 |t|^2 ,\quad \langle \partial_t \hat{F} ,t\rangle\geq d_2\left(|t|^2-|s|^p\right).\]
\end{itemize}

Here, $\langle \partial_s \hat{F} ,w\rangle$ is the standard derivative of $\hat{F}$, that is
$$\langle \partial_s \hat{F} ,w\rangle:=\lim\limits_{h\rightarrow 0} \frac{\hat{F}(s+hw,t)-\hat{F}(s,t)}{h},\ \langle \partial_t \hat{F} ,w\rangle:=\lim\limits_{h\rightarrow 0} \frac{\hat{F}(s ,t+hw)-\hat{F}(s,t)}{h}.$$
Here are some examples to which these assumptions apply.
\begin{exams*}
 Set $\hat{F}(s,t)= |s|^p+ g(|t|)|t|^2$, with $2 < p < 3$, where
\begin{align*}
g(h)=\begin{cases}
\frac{3}{2}\left(1+\ln h\right)+b_2, &\text{if} \quad h\geq 1,\\
\frac{1}{2}\left(1+h+h^2\right)+b_2,& \text{if}\quad 0\leq h\leq 1.
\end{cases}
\end{align*}
In general, if $g:\mathbb{R}_{\geq 0}\rightarrow \mathbb{R}_{\geq 0}$ satisfies
\[0\leq g'(h)\leq \frac{g(h)}{h}-\frac{b_2}{h}\leq \frac{a_2}{h^{3-p}}+\frac{a_2-b_2}{h},\]
then $\hat{F}(s,t)=|s|^p+g(|t|)|t|^2$ satisfies assumptions $(F_1)$-$(F_4)$.
 \end{exams*}

Now, we are ready to state our first main theorem concerning existence and nonexistence of solitary wave solutions with different frequencies of \eqref{eq1.1}. 
\begin{theorem}\label{mainthm}
Let $(F_1)-(F_4)$ be satisfied with $p\in (2,3)$. 
\begin{itemize}
\item[(1)] If $\omega \in [m ,\infty)$, then the nonlinear Dirac equation \eqref{eq1.1} possesses only trivial solution $u\equiv 0$ in $H^1(\mathbb{R}^3,\mathbb{C}^4)$.
\item[(2)] If $\omega\in [-m ,m )$, then the nonlinear Dirac equation \eqref{eq1.1} possesses at least one nontrivial solution $u\not\equiv 0$ in $H^1(\mathbb{R}^3,\mathbb{C}^4)$.
\end{itemize}
\end{theorem}

\begin{figure}[ht]
  \centering
\includegraphics[scale=0.6]{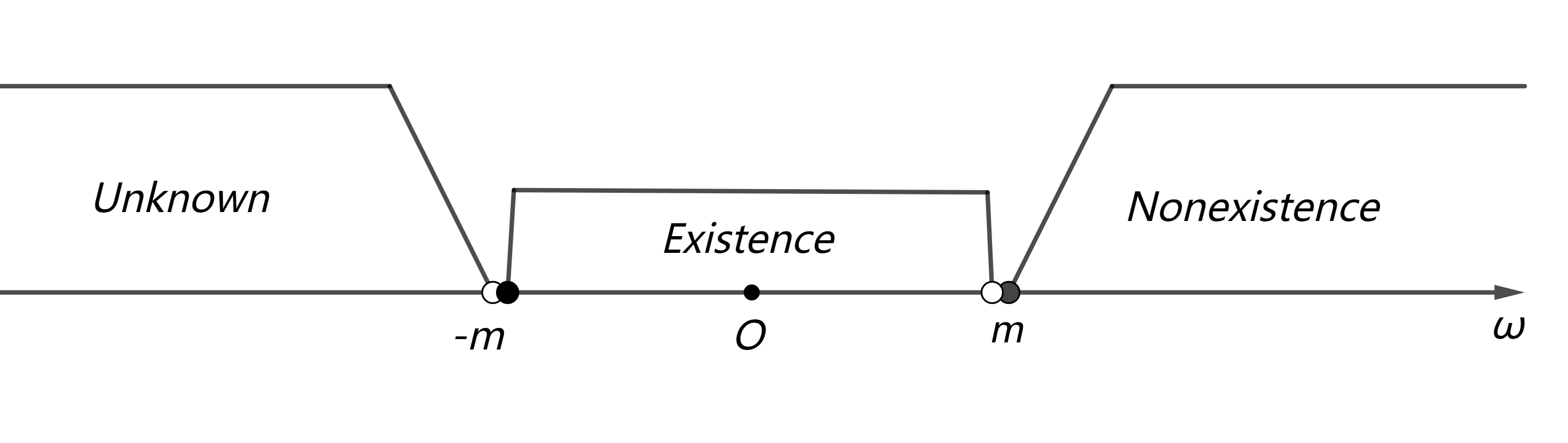}
\caption{Existence and nonexistence results with respect to frequency $\omega$.}
\label{fig:graph}
\end{figure}

\newpage

It is worth mentioning that in \cite{MR2892774}, the Soler model in (1+1) dimension possesses a solitary wave with frequency $-\omega$ when there is a solitary wave with frequency $\omega\in \mathbb{R}$. Our results describe a different phenomenon, that is the solitary wave solution exists when $\omega=-m$, but not when $\omega=m$.
These results rely on the spectrum of the linear operator $D-\omega$.
When $\omega \in (-m, m)$, we have a spectral gap in the linear operator $D-\omega$, and gap solitons can be obtained. However, when $\omega \geq m$,
there are no solitary wave solutions and when $\omega = -m$, there is a solitary wave solution.
It is interesting to note that solitary wave solutions can exist when the frequency lies on the left boundary point of the spectrum, but not on the right boundary point.
From a physical perspective, this implies an asymmetry phenomenon in the sense that, while the propagation of a gap soliton with respect to time is symmetric,
its propagation at the boundary points is asymmetric, and if you observe from front to back, the critical state can be seen, but if you observe from back to front, you may not be able to see it. This is something that would not happen in an idealized state.
When the nonlinear medium has a certain different effect on positive and negative spaces, it results in an asymmetrical state. 


\begin{definition}
    The number $\omega_0\in \mathbb{R}$ is said to be a bifurcation point on the left for the nonlinear Dirac equation if there exists a sequence of nontrivial solutions $\{(\omega, u_\omega)\}\subset \mathbb{R}\times E$ satisfying 
\[Du_\omega-\omega u_\omega=F_u(u_\omega),\quad  \|u_\omega\| \rightarrow 0, \quad \omega \rightarrow \omega_0-.\]
\end{definition}
It is natural to ask 
\begin{question}\label{question1}
   Is $\omega=m$ a bifurcation point on the left for the nonlinear Dirac equation mentioned in Theorem \ref{mainthm}? If so, at what rate do the corresponding sequences converge to zero?
\end{question}
We give a positive answer to Question \ref{question1} in the following theorem:

\begin{theorem}\label{mainthm2}
   Let $(F_1)-(F_4)$ be satisfied with $p\in (2,8/3)$. Then $m$ is a bifurcation point on the left for the nonlinear Dirac equation \eqref{eq1.1}. Moreover, we have we have 
    \[\|u_\omega\| \leq C(m-\omega)^{\frac{8-3p}{2(p-2)}}\rightarrow 0,\quad \text{as}\quad \omega\rightarrow m-.\]
\end{theorem}
\begin{figure}[ht] 
    \centering
    \includegraphics[width=0.5\linewidth]{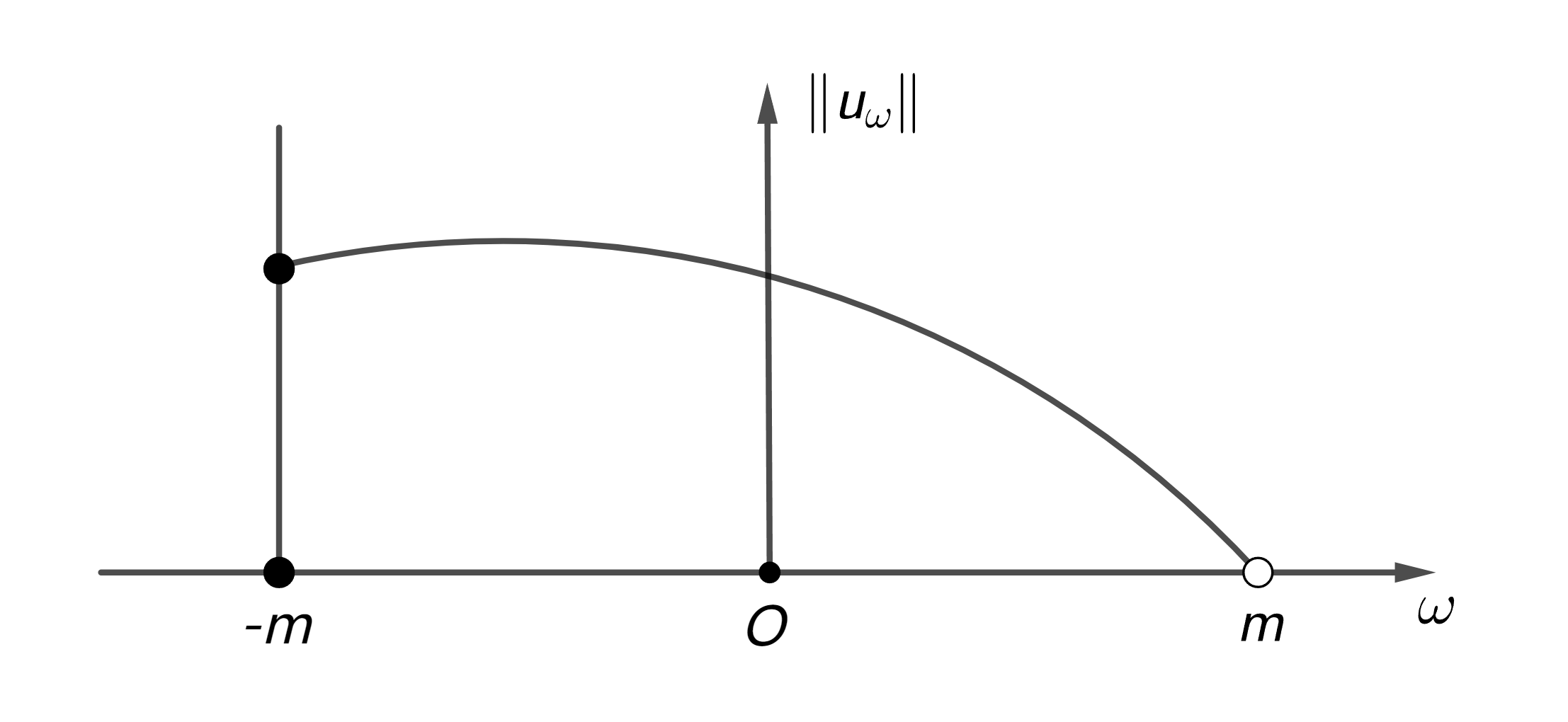}
    \caption{Bifurcation branch in the parameter $\omega$.}
    \label{fig:2}
\end{figure}
For bifurcation into the spectral gap, Heinz, K\"uper and Stuart applied some of their previous results on bifurcation theory to study the nonlinear Schr\"odinger equation, see \cite{HEINZ1992341}. They show that the right boundary point of the spectral gap is a bifurcation point. In \cite{KUPPER1992893,LRT,stuart83}, it is also shown that no bifurcation from 0 can occur at the left boundary point of the spectral gap. Based on the above discussions, we show that the point $-m$ can be reached by the bifurcation branch,
it leads to the following question.
\begin{question}\label{q-1}
What happens with the branch on the left 
 of $\omega =-m$: can it be extended a little bit to the left of $-m$ into the spectrum?
\end{question}

Unfortunately, we know nothing about the situation when $\omega<-m$. However, judging from the trend on Figure \ref{fig:2}, we conjecture that this branch will extend to the left of $-m$ and disappear at some point.

There are many difficulties from a mathematical perspective. When trying to obtain a nonexistence result for $\omega \geq m$, Pohozaev's identity of \eqref{eq1.1} needs to be established.
Finding the gap soliton requires the use of a critical point theorem for  the strongly indefinite problem.
 We need to analyze the topological properties and geometric (linking) structures of the corresponding functional.
However, when the frequency equals $-m$, the second linking condition fails and $\|u\|$ cannot be controlled by $\|u^+\|$.
Unfortunately, previous methods on the {\it Spectrum Zero Problem} cannot be applied to \eqref{eq1.1}, as the spectrum of the linear operator is unbounded from both the top and bottom,
there is no proper embedding of the current working space, and $L^2$-boundedness of $(C)_c$-sequences cannot be proven.
To tackle this issue, a new approach has been developed here. It involves constructing a $(PS)_c$-sequence through perturbation,
demonstrating the uniform bounded properties of the sequence and energy, and finally taking the limit of the sequence.

\medskip

\noindent{\bf Outline of the paper.} This paper is organized as follows. In section \ref{sec2}, we present some preliminary notions on the Dirac operator and some basic facts which will be used later.
And in section \ref{sec3}, we show the nonexistence result for the case $\omega\geq m$. Then we investigate the existence of solitary solutions when $\omega\in (-m,m)$ in section \ref{sec4} and  show the proof of the existence result when $\omega=-m$ in section \ref{sec5}. At last, we give a proof of Theorem \ref{mainthm2} in section \ref{sec6}.

\par \medskip

\noindent{\bf Acknowledgments.}
The first author was supported by the National Natural Science Foundation of China (NSFC 12201625), the second author was supported by NSFC 12341102 and the Beijing Natural Science Foundation (Z210002). The third author was supported by RIMS, the Research Institute for Mathematical Sciences, an International Joint Usage/Research Center located at Kyoto University.

\section{Preliminaries and Technical Lemmas}\label{sec2}

\subsection{Notations}\label{sec2.1}
In this paper, we denote $\alpha\cdot\nabla=\sum\limits_{k=1}^3 \alpha_k\partial_k$, and $\partial_k=\frac{\partial}{\partial x_k}$, where the
matrices $\alpha_1$, $\alpha_2$, $\alpha_3$ and $\beta$ are $4\times 4$ Dirac matrices,
 \begin{align*}
\alpha_k=\begin{pmatrix}
0 &\sigma_k\\
\sigma_k &0
\end{pmatrix},\beta=\begin{pmatrix}
I_2 &0\\
0 &-I_2
\end{pmatrix},
\end{align*}
with
\begin{align*}
\sigma_1=\begin{pmatrix}
0 &1\\
1&0
\end{pmatrix}, \sigma_2=\begin{pmatrix}
0 &-i\\
i&0
\end{pmatrix}, \sigma_3=\begin{pmatrix}
1&0\\
0&-1
\end{pmatrix}.
\end{align*}
It is clear that the Dirac matrices satisfy the following Clifford relations:
\begin{align*}
\begin{cases}
\alpha_k^*=\alpha_k,\ \beta^*=\beta,\\
\alpha_k\alpha_l+\alpha_l\alpha_k=2\delta_{kl}I_4,\
\alpha_k\beta+\beta\alpha_k=0,\
\beta^2=I_4.
\end{cases}
\end{align*}
The gamma matrices are also $4\times 4$ matrices which satisfy the following anti-commutation relations:
\[\{\gamma^\mu,\gamma^\nu\}=\gamma^\mu\gamma^\nu+\gamma^\nu\gamma^\mu=2g^{\mu\nu}I_4,\]
where $g^{\mu\nu}$ is the Minkowski metric element, and the indices $\mu,\nu$ run over $0,1,2,3$. They can be formed as
 \begin{align*}
\gamma^\mu=\begin{pmatrix}
0 &\sigma_\mu\\
-\sigma_\mu &0
\end{pmatrix},\gamma^0=\begin{pmatrix}
I_2 &0\\
0 &-I_2
\end{pmatrix},
\end{align*}
 we have $\gamma^\mu=\beta\alpha_\mu$, if $\mu=1,2,3$, and $\gamma^0=\beta$. Another matrix plays an important role in the parity transformation, the $\gamma^5$ matrix, which is defined by
 \[ \gamma^5=i\gamma^0\gamma^1\gamma^2\gamma^3=\begin{pmatrix}
 0 & I_2\\
 I_2 & 0
 \end{pmatrix}.
 \]
 The Dirac adjoint of $\psi$ is defined as
 \[\bar{\psi}=\psi^\dagger \gamma^0.\]
\subsection{Technical Lemmas}

Let $D=-i \alpha\cdot \nabla+m \beta$ be the free Dirac operator, which is a self-adjoint operator on $L^2(\mathbb{R}^3,\mathbb{C}^4)$ with domain $H^1(\mathbb{R}^3,\mathbb{C}^4)$ and formal domain $H^{1/2}(\mathbb{R}^3,\mathbb{C}^4)$. It is clear that $$\sigma (D)=\sigma_c(D)=\mathbb{R}\setminus (-m ,m),$$
Set $E:=\text{dom}(|D|^{1/2})$ equipped with the following real inner product
\[ ( u,v) :=\Re (|D|^{1/2}u,|D|^{1/2}v)_{L^2},\]
where $(\cdot, \cdot)_{L^2}$ is the complex inner product in $L^2(\mathbb{R}^3,\mathbb{C}^4)$. Then we have

\begin{lemma}\label{lem2.1}
$E\cong H^{1/2}(\mathbb{R}^3,\mathbb{C}^4)$, hence $E$ embeds continuously into $L^p$, for all $p\in [2,3]$ and compactly into $L^p_{loc}$ for all $p\in [1,3)$.
Moreover, we have
\[m\|u\|^2_{L^2}\leq \|u\|^2,\quad \forall u\in E.\]
\end{lemma}

We recall a lemma due to P.L. Lions. The proof of this lemma can be found in \cite{MR4320890} or \cite{Kryszewski}.
\begin{lemma}\label{Lem3.3}
	Fix $r>0$ and $s \in\left[2,3\right)$. If $\left\{u_{n}\right\}$ is bounded in $H^{1/2}\left(\mathbb{R}^{3}, \mathbb{C}^4\right)$ and if
	$$
	\sup _{y \in \mathbb{R}^{3}} \int_{B(y, r)}\left|u_{n}\right|^{s} d x \rightarrow 0 \quad \text { as } n \rightarrow \infty,
	$$
	then $u_{n} \rightarrow 0$ in $L^{t}\left(\mathbb{R}^{3}, \mathbb{C}^4\right)$ for any $t \in\left(2,3\right)$.
\end{lemma}

The proof of the following Lemma can be found in \cite{MR4610809}. It is based on Fourier transform and a Mikhlin-type Fourier multiplier theorem.

\begin{lemma}\label{projection}
Let $E^\pm_q:=E^\pm\cap L^q$ for $q\in (1,\infty)$. Then there holds
\[L^q=\text{cl}_q E^+_p\oplus \text{cl}_q E^-_q,\]
where $\text{cl}_q$ denotes the closure with respect to the norm in $L^q$. That is, there exists $\tau_q>0$ for every $q\in (1,\infty)$ such that
\[\tau_q\|u^\pm\|_{L^q}\leq \|u\|_{L^q}, \ \forall u\in E\cap L^q.\]
\end{lemma}

Denote $\Psi(u)=\displaystyle\int_{\mathbb{R}^3} F( u)dx=\displaystyle\int_{\mathbb{R}^3} \hat{F}( u^+,u^-)dx$, then we have
\begin{lemma}\label{lem2.4}
Under the assumptions of Theorem \ref{mainthm}, the following conclusions hold:
\begin{itemize}
\item[$(1)$] There exists $a_1, a_2>0$, such that
$$a_1 \left(\|u^+\|_{L^p}^p +\|u^-\|_{L^2}^2\right) \leq \Psi(u) \leq a_2\left(\|u^+\|_{L^p}^p+ \|u^-\|_{L^p}^p+ \|u^-\|_{L^2}^2\right) .$$
\item[$(2)$] There exists $b_1, b_2>0$, such that
$$ 2\Psi(u)+b_1  \|u^+\|_{L^p}^p \leq  \langle\Psi'(u),u\rangle\leq 3\Psi(u)-b_2\left( \|u^+\|_{L^p}^p +\|u^-\|_{L^2}^2\right).$$
\item[$(3)$] There exists $c_1>0$, $d_2\geq d_1>0$, such that
\[\langle\Psi'(u ),u ^+\rangle\leq c_1 \|u^+\|_{L^p}^p+ d_1\|u^-\|_{L^2}^2 ,\quad
\langle \Psi'(u),u^-\rangle \geq d_2\left(\|u^- \|_{L^2}^2-\|u^+\|_{L^p}^p \right).\]

\end{itemize}
\end{lemma}

\begin{proof}
\begin{itemize}
\item[$(1)$] It follows directly from the definition of $\Psi$ and $(F_2)$ that
$$a_1\left(\|u^+\|_{L^p}^p +\|u^-\|_{L^2}^2\right)  \leq \Psi(u) \leq a_2\left(\|u^+\|_{L^p}^p+ \|u^-\|_{L^p}^p+ \|u^-\|_{L^2}^2\right) .$$
\item[$(2)$] It follows directly from $E\cong E^+\oplus E^-$ that
\begin{align*}
\langle \Psi'(u),u\rangle &= \int_{\mathbb{R}^3}\left\langle F_u(u),u\right\rangle dx\\
&=\lim_{h\rightarrow 0} \int_{\mathbb{R}^3} \frac{F((1+h)u)-F(u)}{h}dx\\
&=\lim_{h\rightarrow 0} \int_{\mathbb{R}^3} \frac{\hat{F}((1+h)u^+,(1+h)u^-)-\hat{F}(u^+, u^-)}{h}dx\\
&=\int_{\mathbb{R}^3} \langle \partial_{s}\hat{F}( u^+,u^-),u^+\rangle +\langle \partial_{t} \hat{F}(u^+,u^-),u^-\rangle dx.
\end{align*}
Here, $$\langle \partial_s \hat{F}(u^+,u^-),u^+\rangle:=\lim\limits_{h\rightarrow 0} \frac{\hat{F}((1+h)u^+ ,u^-)-\hat{F}(u^+,u^-)}{h},$$ $$\langle \partial_t \hat{F}(u^+,u^-),u^-\rangle:=\lim\limits_{h\rightarrow 0} \frac{\hat{F}(u^+ ,(1+h)u^-)-\hat{F}(u^+,u^-)}{h}.$$
From $(F_3)$, we obtain
$$  2\Psi(u)+b_1  \|u^+\|_{L^p}^p \leq  \langle\Psi'(u),u\rangle\leq 3\Psi(u)-b_2\left( \|u^+\|_{L^p}^p +\|u^-\|_{L^2}^2\right).$$
\item[$(3)$] Similar to $(2)$, we have
\[\langle\Psi'(u ),u ^+\rangle=\int_{\mathbb{R}^3} \langle \partial_{s}\hat{F}( u^+,u^-),u^+\rangle dx, \]
\[\langle\Psi'(u ),u ^-\rangle=\int_{\mathbb{R}^3} \langle \partial_{t}\hat{F}( u^+,u^-),u^-\rangle dx. \]
Therefore, by $(F_4)$, we have
\[\langle\Psi'(u ),u ^+\rangle\leq c_1 \|u^+\|_{L^p}^p+ d_1\|u^-\|_{L^2}^2 ,\quad
\langle \Psi'(u),u^-\rangle \geq d_2\left(\|u^- \|_{L^2}^2-\|u^+\|_{L^p}^p \right).\]
This ends the proof.
\end{itemize}
\end{proof}

\section{Nonexistence results of (NDE) with large frequencies}\label{sec3}

In this section, we first derive a Pohozaev identity for nonlinear Dirac equations. Then we show the nonexistence result with $\omega$ large enough by using Pohozaev's identity and a variational identity.
In this section, $\langle \cdot, \cdot\rangle$ represents the complex inner product in $\mathbb{C}^4$, $\Re f$ ($\Im f$) represents the real (imaginary) part of $f$.

\begin{proposition}[Pohozaev's identity]
	Under the assumptions of Theorem \ref{mainthm}, if $u\in H^1(\mathbb{R}^3, \mathbb{C}^4)$ is a solution of \eqref{eq1.1},
then $u$ satisfies the following Pohozaev identity:
\begin{equation}\tag{3.1}\label{Pohozaev}
\begin{aligned}
\int_{\mathbb{R}^{3}}\langle-i \alpha \cdot \nabla u, u\rangle dx +\frac{3}{2} \int_{\mathbb{R}^{3}}\left(\left\langle m \beta u, u\right\rangle-\omega|u|^{2}\right) dx
=3\int_{\mathbb{R}^{3}} F(u)  dx.
\end{aligned}
\end{equation}
\end{proposition}

\begin{proof}
Taking the scalar product of \eqref{eq1.1} with $x\cdot \nabla u,$ we obtain for all $R>0$,
\begin{equation*}
\Re\int_{B_R}\langle -i \alpha\cdot \nabla u, x\cdot \nabla u\rangle dx=\Re\int_{B_R}\langle -m \beta u+\omega u+F_u(u), x\cdot \nabla u\rangle dx,
\end{equation*}
where $B_R:=\{x\in \mathbb{R}^3:|x|<R\}$, $\partial B_R:=\{x\in \mathbb{R}^3:|x|=R\}$.
We denote the outward unit vector to $\partial B_R$ by $n=(n_1,n_2,n_3)$. Integrating by parts in $B_R$, we obtain
\begin{equation*}\label{V}
\begin{aligned}
&\Re\int_{B_{R}}  \langle m\beta u,x \cdot \nabla u\rangle d x-\Re\int_{B_{R}} \langle  \omega  u,x \cdot \nabla u\rangle dx
\\
=&\frac{1}{2} \int_{B_{R}}  m  \sum_{i=1}^{3} x_{i}\frac{\partial}{\partial x_{i}}\langle \beta u, u\rangle d x
-\frac{1}{2} \int_{B_{R}}\omega \sum_{i=1}^{3} x_{i}\frac{\partial}{\partial x_{i}}\left(|u|^{2}\right) d x
\\
=&\frac{1}{2} \int_{\partial B_{R}}\left(m\langle \beta u, u\rangle -\omega  |u|^{2} \right) \sum_{i=1}^{3} x_{i} n_{i} d S
-\frac{3}{2} \int_{B_{R}} \left(m\langle \beta u, u\rangle-\omega \left|u\right|^{2} \right)d x
\\	=&\frac{ R}{2} \int_{\partial B_{R}}\left( m  \langle \beta u, u\rangle-\omega |u|^{2} \right)  d S -\frac{3}{2} \int_{B_{R}} \left( m \langle \beta u, u\rangle -\omega |u|^{2}\right) d x.
	\end{aligned}
\end{equation*}
Similarly,
\begin{equation*}\label{W}
\begin{aligned}	\Re\int_{B_{R}} \langle F_u( u),  x \cdot \nabla  u\rangle d x
=&\sum_{i=1}^{3}\Re\int_{B_{R}}  \langle F_u( u), x_{i} \frac{\partial}{\partial x_{i}}u \rangle d x\\
=&\sum_{i=1}^{3}\int_{B_{R}}  \sum_{i=1}^3 x_i\frac{\partial}{\partial x_i} F(u) d x\\
=& \int_{\partial B_{R}}  F(u)\sum\limits_{i=1}^3x_in_idS-3\int_{B_{R}} F ( u) dx \\
=&\int_{\partial B_{R}} F(u) RdS-3\int_{B_{R}}F(u) dx
	\end{aligned}
\end{equation*}
Moreover,
\begin{equation*}\label{D}
	\begin{aligned}
		& \Re\int_{B_{R}}\langle -i \alpha\cdot \nabla u,  x \cdot \nabla u\rangle d x\\
		=&\frac{1}{2} \Re\int_{\partial B_{R}}\left(\langle -i \alpha\cdot \nabla u,  u\rangle  R + \langle  x \cdot \nabla u, -i \alpha\cdot n u\rangle \right)d S-\int_{B_{R}}\langle -i \alpha\cdot \nabla u,  u\rangle d x
	\end{aligned}
\end{equation*}
Therefore, since $\langle \beta u, u\rangle, |u|\in \mathbb{R}$,  then we obtain
\begin{equation*}
	\begin{aligned}
		&\left|\int_{B_{R}}\langle-i \alpha \cdot \nabla u, u\rangle d x +\frac{3}{2} \int_{B_{R}}\left(\left\langle m   \beta u, u\right\rangle-\omega|u|^{2}\right) d x-3 \int_{B_{R}} F(u) d x    \right|\\
		=& \bigg|\frac{1}{2}\Re \int_{\partial B_{R}}\left(\langle -i \alpha\cdot \nabla u,  u\rangle R+ \langle  x \cdot \nabla u, -i \alpha\cdot n u\rangle \right)d S -  \int_{\partial B_{R}}F(u) R d S\\
		 &+\frac{R}{2} \int_{\partial B_{R}}\left( m  \langle \beta u, u\rangle-  \omega  |u|^{2} \right)  d S \bigg|\\
		\leq &  RC \int_{\partial B_{R}}\left(|\nabla u||u|+|u|^2+|u^+|^p+|u^-|^p\right)dS,
	\end{aligned}
\end{equation*}
where $C>0$ is a constant.
Then, letting $R\rightarrow \infty$, since $u\in H^1(\mathbb{R}^3, \mathbb{C}^4)$, the last term of the above inequality vanishes. This ends the proof.
\end{proof}
\medskip

 The above Pohozaev identity can be used to show the nonexistence results for the nonlinear Dirac equations.
Next, we show the proof of Theorem \ref{mainthm} (1).
\medskip
\\
{\bf Proof of Theorem \ref{mainthm} (1).}
	We argue by contradiction and suppose that there exists $u\neq 0$ satisfying \eqref{eq1.1} in $H^1(\mathbb{R}^3, \mathbb{C}^4)$.
	Then, 	
	\begin{equation*}
		\int_{\mathbb{R}^{3}}\left\langle-i \alpha\cdot \nabla u,u \right\rangle dx=\int_{\mathbb{R}^{3}}\left( -\langle m \beta u,u\rangle+\omega |u|^2  +\langle F_u(u),u\rangle\right)dx,
	\end{equation*}
which, together with Pohozaev's identity \eqref{Pohozaev}, implies that
	\begin{equation*}
		\begin{aligned}
			 \frac{1}{2}\int_{\mathbb{R}^{3}}\left( \langle m \beta u,u\rangle-\omega |u|^2\right)dx=  3\int_{\mathbb{R}^{3}}F(u)dx -\int_{\mathbb{R}^{3}}\left\langle F_u( u) ,u\right\rangle dx.
		\end{aligned}
	\end{equation*}
	By the assumption that $\omega\geq m$, we have
 \begin{align*}
     \int_{\mathbb{R}^{3}}\left( \langle m \beta u,u\rangle-\omega |u|^2\right)dx=\int_{\mathbb{R}^{3}}\left( (m-\omega)\sum_{k=1}^2|u_k|^2-(m+\omega)\sum_{k=3}^4|u_k|^2\right)dx \leq 0.
 \end{align*}
It follows directly that
 \begin{equation*}
 		\begin{aligned}
3\Psi(u)-\langle \Psi'(u),u\rangle
 		\leq   0.
 	\end{aligned}
 \end{equation*}
Then, Lemma \ref{lem2.4} (2) implies that
\[\|u^+\|_{L^p}^p+\|u^-\|_{L^2}^2=0.\]
Therefore, we have $u =0$ in $H^1(\mathbb{R}^3,\mathbb{C}^4)$, which is a contradiction. $\hfill\Box$

\section{Existence results of (NDE) in spectral gap}\label{sec4}
In this section, we show the existence results of \eqref{eq1.1} when $\omega\in (-m,m)$.
Recall that $D=-i\alpha\cdot \nabla+m\beta$, and $E=\mathscr{D}(|D|^{1/2})$ is the domain of the self-adjoint operator $|D|^{1/2}$, which is a real Hilbert space equipped with the real inner product
\[(u,v)=\Re(|D|^{1/2}u,|D|^{1/2}v)_{L^2},\]
and the induced norm $\|u\|=(u,u)^{1/2}$. By using the Fourier transform, we have
$$
\frac{m-|\omega|}{m}\left\|u^{+}\right\|^2 \leq \left(\left\|u^{+}\right\|^2 \pm \omega\left\|u^{+}\right\|_{L^2}^2\right) \leq  \frac{m+|\omega|}{m}\left\|u^{+}\right\|^2,
$$
and
$$
\frac{m-|\omega|}{m}\left\|u^{-}\right\|^2 \leq \left(\left\|u^{-}\right\|^2 \pm \omega\left\|u^{-}\right\|_{L^2}^2\right) \leq \frac{m+|\omega|}{m}\left\|u^{-}\right\|^2 .
$$
Now we introduce a functional on $E$ which is defined by
\[\Phi_\omega(u)=\frac{1}{2}\left(\|u^+\|^2-\|u^-\|^2\right)-\frac{\omega}{2}\|u\|_{L^2}^2-\Psi(u),\]
where $\Psi(u)=\displaystyle\int_{\mathbb{R}^3} F( u(x))dx$. Thus, the existence of weak solutions of \eqref{eq1.1} is equivalent to the existence of critical points of $\Phi_\omega$ on $E$. To find the critical point of $\Phi_\omega$, we need to use the critical point theorem of strong indefinite type. First, we show that the functional $\Phi_\omega\in \mathcal{C}^1(E,\mathbb{R})$ possesses the following topological properties and geometric structures.
\par \bigskip

\noindent{\bf $*$ Topological Properties of $\Phi_\omega$.}

\begin{lemma}\label{psi}
Under the assumptions of Theorem \ref{mainthm}, the following conclusions hold:
\begin{itemize}
\item[(1)] $\Psi\in \mathcal{C}^1(E,\mathbb{R} )$ is bounded from below,
\item[(2)] $\Psi$ is weakly sequentially lower-semicontinuous,
\item[(3)] $\Psi'$ is weakly sequentially continuous.
\end{itemize}
\end{lemma}
\begin{proof}
\begin{itemize}
\item[(1)] It is clear that $\Psi\in \mathcal{C}^1(E,\mathbb{R})$. By Lemma \ref{lem2.4} (1), we have for every $u\in E$,
\[\Psi(u)\geq a_1 \left(\|u^+\|_{L^p}^p+\|u^-\|_{L^2}^2\right)\geq 0.\]
Thus, $\Psi$ is bounded from below.
\item[(2)] Let $u_k\in E$ with $u_k\rightharpoonup u$ in $E$. Hence, $u_k^+\rightharpoonup u^+$, $u_k^-\rightharpoonup u^-$ in $E$. Thus, for almost every $x\in \mathbb{R}^3$, we have
$ F(u_k(x))\rightarrow F (u(x)).$
By Fatou's lemma, we get
\begin{align*}
\Psi(u)=\int_{\mathbb{R}^3} F(u(x))dx\leq \liminf_{k\rightarrow \infty} \int_{\mathbb{R}^3} F(u_k(x)) dx,
\end{align*}
which implies that $\Psi$ is weakly sequentially lower-semicontinuous.
\item[(3)] It is clear that for any $\phi\in \mathcal{C}_0^\infty(\mathbb{R}^3,\mathbb{C}^4)$, we have
\begin{align*}
\langle \Psi'(u),\phi\rangle= \Re \int_{\mathbb{R}^3} \langle F_u(u),\phi\rangle  dx.
\end{align*}
The Lebesgue dominated convergence theorem yields $$\langle \Psi'(u_k),\phi\rangle\rightarrow \langle \Psi'(u),\phi\rangle,\quad \text{for any }\ \phi\in \mathcal{C}_0^\infty(\mathbb{R}^3,\mathbb{C}^4),$$
which implies that $\Psi'$ is weakly sequentially continuous.
\end{itemize}
\end{proof}

For a functional $\Phi_\omega \in \mathcal{C}^1(E, \mathbb{R})$, we write $$\Phi_{ a}:=\{u \in E: \Phi_\omega (u) \geq  a\},\quad  \Phi^b:=\{u \in E: \Phi_\omega(u) \leq b\}, \quad \text{and}\quad \Phi_a^b:=\Phi_a \cap \Phi^b.$$
Recall that a sequence $\left\{u_k\right\}\subset E$ is said to be a $(P S)_c$-sequence if $\Phi_\omega\left(u_k\right) \rightarrow c$ and $\Phi_\omega^{\prime}\left(u_k\right) \rightarrow 0$.
It is called a $(C)_{c}$-sequence if $\Phi_\omega\left(u_k\right) \rightarrow c$ and $\left(1+\left\|u_k\right\|\right) \Phi_\omega^{\prime}\left(u_k\right) \rightarrow 0$.
The functional $\Phi$ is said to satisfy the $(P S)_c$-condition (or $(C)_c$-condition) if any $(P S)_c$-sequence (or $(C)_c$-sequence, respectively) has a convergent subsequence.
Set $X=E^-$, $Y=E^+$. It is clear that $X$ is separable and reflexive, and we fix a dense subset $\mathcal{S} \subset X^*$. For each $s \in \mathcal{S}$ there is a semi-norm on $E$ defined by
$$
p_s: E \rightarrow \mathbb{R}, \quad p_s(u)=|s(x)|+\|y\| \quad \text { for } u=x+y \in X \oplus Y .
$$
We denote by $\mathcal{T}_{\mathcal{S}}$ the induced topology. Let $w^*$ denote the weak*-topology on $E^*$.

\begin{lemma}\label{topo} Under the assumptions of Theorem \ref{mainthm}, the following conclusion holds:
\begin{itemize}
\item[(1)] $\Phi_c$ is $\mathcal{T}_{\mathcal{S}}$-closed and $\Phi_\omega':(\Phi_c,\mathcal{T}_{\mathcal{S}})\rightarrow (E^*,w^*)$ is continuous.
\item[(2)] There is $\zeta >0$, such that for any $c>0$,
\[\|u\|<\zeta \|u^+\|,\quad \forall u\in \Phi_c.\]
\end{itemize}
\end{lemma}

\begin{proof}
\begin{itemize}
\item[(1)] It follows directly from Lemma \ref{psi} and Proposition 4.1 in \cite{MR2255874}.

\item[(2)]  Assume by contradiction that for some $c>0$, there exists a sequence $\left\{u_k\right\} \subset$ $\Phi_c$, and $\left\|u_k\right\|^2 \geq k\left\|u_k^{+}\right\|^2$.
It follows directly that
$$0 \geq(2-k)\left\|u_k^{-}\right\|^2 \geq(k-1)\left(\left\|u_k^{+}\right\|^2-\left\|u_k^{-}\right\|^2\right).$$
Therefore, we get
$$
\Phi_\omega\left(u_k\right)=\frac{1}{2}\left(\|u_k^+\|^2-\|u_k^-\|^2\right)+\frac{\omega}{2}\|u_k\|_{L^2}^2-\Psi(u_k) \leq C\left(\|u_k^+\|^2-\|u_k^-\|^2\right)-\Psi(u_k)\leq 0,
$$
which is a contradiction.
\end{itemize}
\end{proof}
\noindent{\bf $*$ Geometric Structures of $\Phi_\omega$.}
\medskip

The existence of critical points depends on the local geometric structure of the functional. By proving the following Linking lemma, we can further construct a $(C)_c$-sequence.
\begin{lemma}\label{geo}
Under the assumptions of Theorem \ref{mainthm}, the following conclusion holds:
\begin{itemize}
\item[(1)] There exists $\rho$, such that $\kappa:=\inf \Phi_\omega  (\partial B_\rho\cap E^+)>0$.
\item[(2)] For some $e\in E^+$, with $\|e\|=1$, there exists $R$, $\sigma>0$ (independent of $\lambda$), such that
\[\sup\Phi_\omega(E_e)<\sigma,\quad \sup\Phi_\omega(E_e\setminus B_R)\leq 0,\]
where $B_R:=\{u\in E_e: \|u\|\leq R\}$, $E_e:=E^-\oplus \mathbb{R}e$.
\end{itemize}
\end{lemma}
\begin{proof}
\begin{itemize}
\item[(1)] For any $u\in E^+$, we have
\[\Phi_\omega (u)=\frac{1}{2}\|u\|^2-\frac{\omega}{2}\|u\|_{L^2}^2-\Psi(u),\]
By the assumptions on $\Psi$, we have
$$\Psi(u)\leq a_2 \|u\|_{L^p}^p \leq C\|u\|^p.$$ Thus, we have
\[\Phi_\omega (u) \geq \frac{1}{2}\|u\|^2-\frac{\omega}{2}\|u\|_{L^2}^2-C\|u\|^p\geq \frac{m-|\omega|}{2m}\|u\|^2 -C\|u\|^p.\]
Therefore, we can choose $\rho$ small, such that $\kappa>0$.
\item[(2)] Let $u=u^-+te\in E_e$, then we have
\[\Psi(u)\geq a_1\left( \|u^+\|_{L^p}^p+\|u^-\|_{L^2}^2\right) =a_1\left( t^p\|e\|_{L^p}^p+\|u^-\|_{L^2}^2\right) .\]
Therefore,
\begin{align*}
\Phi_\omega (u)&=\frac{t^2}{2}-\frac{1}{2}\|u^-\|^2-\frac{\omega}{2}\|u^-+te\|^2_{L^2}-\Psi(u)\\
&\leq \left(\frac{1}{2}-\frac{m}{2}\|e\|_{L^2}^2\right) t^2-\frac{1}{2}\left(\|u^-\|^2+\omega\|u^-\|_{L^2}^2\right)-a_1t^p\|e\|_{L^p}^p\\
&\leq C(t^2-t^p).
\end{align*}
Here, we choose $e\in E^+$ such that $\|e\|_{L^p}\neq 0$. Therefore, there exists $\sigma>0$, such that
$\sup\Phi_\omega (E_e)<\sigma.$ Based on the previous argument, it is clear that we can choose $R$ large, such that
\[\sup\Phi_\omega(E_e\setminus B_R)\leq 0.\]
\end{itemize}

\end{proof}

\noindent{\bf $*$ The Cerami sequence.}

\begin{lemma}\label{4.3}
Under the assumptions of Theorem \ref{mainthm}, any $(C)_c$-sequence of $\Phi_\omega$ is bounded.
\end{lemma}

\begin{proof}
Let $\{u_k\}$ be a $(C)_c$-sequence of $\Phi_\omega$. Since
\[\Phi_\omega(u_k)-\frac{1}{2}\langle \Phi_\omega'(u_k),u_k\rangle=\frac{1}{2}\langle \Psi'(u_k),u_k\rangle-\Psi(u_k)\geq \frac{b_1}{2}\|u_k^+\|_{L^p}^p .\]
We have $\left\{\|u_k^+\|_{L^p}\right\}$ is bounded. Observe that
\begin{align*}
\langle \Phi_\omega'(u_k),u_k^+-u_k^-\rangle =\|u_k\|^2-\omega\left(\|u_k^+\|_{L^2}^2-\|u_k^-\|_{L^2}^2\right)-\langle \Psi'(u_k),u_k^+-u_k^-\rangle,
\end{align*}
and by Lemma \ref{lem2.4} (3), we have
\[\langle \Psi'(u_k),u_k^+-u_k^-\rangle\leq C\left(\|u_k^+\|_{L^p}^p-\|u_k^-\|_{L^2}^2\right)\leq C.\]
Therefore, we get
\[\frac{m-|\omega|}{m}\|u_k\|^2\leq \|u_k\|^2-\omega\left(\|u_k^+\|_{L^2}^2-\|u_k^-\|_{L^2}^2\right)=\langle \Phi_\omega'(u_k),u_k^+-u_k^-\rangle+\langle \Psi'(u_k),u_k^+-u_k^-\rangle .\]
Since $\omega\in (-m,m)$, we obtain $\{u_k\}$ is bounded in $E$.
\end{proof}

\medskip
Based on our prior discussion, we are able to derive a $(C)_c$-sequence by making use of the critical point theorem. Subsequently, we employ this sequence to ascertain a weak solution for \eqref{eq1.1} by invoking the concentration compactness argument. This methodology forms the core essence of the proof for the existence results of \eqref{eq1.1} when $\omega$ falls within the range of $(-m,m)$. This is a scenario we notate as $(2.1)$.

\medskip

\noindent {\bf Proof of Theorem \ref{mainthm} (2.1).} Since the functional $\Phi_\omega\in\mathcal{C}^1(E,\mathbb{R})$ possesses the above topological properties and geometric structures as in Lemma \ref{topo} and \ref{geo}.
By the standard critical point theorem in \cite{MR2255874}, we have $\Phi_\omega$ possesses a $(C)_c$-sequence $\{u_k\}$ with $\kappa\leq c\leq \sup \Phi_\omega(Q)$.
By Lemma \ref{4.3}, $\{u_k\}$ is bounded in $E$. Thus we may assume that $u_k\rightharpoonup u$ in $L^2(\mathbb{R}^3,\mathbb{C}^4)$.

\medskip
{\bf Claim.} For $r>0$ arbitrary there exists a sequence $\left\{y_{k}\right\}$ in $\mathbb{R}^{3}$ and $\eta>0$ such that
\begin{equation*}
	\liminf _{k \rightarrow \infty} \int_{B_r\left(y_{k} \right)}\left|u_{k} \right|^{2} d x \geq \eta .
\end{equation*}
If not, then by Lemma \ref{Lem3.3}, we have $u_{k}\rightarrow 0$ in $L^{t}\left(\mathbb{R}^{3}, \mathbb{C}^4\right)$, for any $t \in\left(2,3\right)$.
Then by Lemma \ref{projection}, we have $u_{k}^\pm\rightarrow 0$ in $L^{t}\left(\mathbb{R}^{3}, \mathbb{C}^4\right)$. Moreover, we have
\begin{align*}
\langle \Phi_\omega'(u_k), u_k^+-u_k^-\rangle =\|u_k \|^2-\omega\|u_k^+\|_{L^2}^2+\omega\|u_k^-\|_{L^2}^2-\langle \Psi'(u_k),u_k^+-u_k^-\rangle,
\end{align*}
and
\[\langle \Psi'(u_k),u_k^+-u_k^-\rangle \leq C\left(\|u_k^+\|_{L^p}^p-\|u_k^-\|_{L^2}^2\right)\leq C\|u_k^+\|_{L^p}^p\rightarrow 0.\]
Therefore, we have $\|u_k^+\|\rightarrow 0$, and $\|u_k^+\|_{L^2}\rightarrow 0$, as $k\rightarrow \infty$. It is clear that
\begin{align*}
\Phi_\omega(u_k)&=\frac{1}{2}\left(\|u_k^+\|^2-\|u_k^-\|^2\right)-\frac{\omega}{2}\|u_k\|_{L^2}^2-\Psi(u_k)\\
&=\frac{1}{2}\left(\|u_k^+\|^2-\omega\|u_k^+\|_{L^2}^2\right)-\frac{1}{2}\left(\|u_k^-\|^2+\omega\|u_k^-\|_{L^2}^2\right)-\Psi(u_k)\\
&\leq \frac{1}{2}\left(\|u_k^+\|^2-\omega\|u_k^+\|_{L^2}^2\right).
\end{align*}
This yields $c=\lim\limits_{k \rightarrow \infty} \Phi_\omega\left(u_{k}\right) \leq 0$, a contradiction.

\medskip

Now we choose $a_{k} \in \mathbb{Z}^{3}$ such that $\left|a_{k}-y_{k}\right|=\min \left\{\left|a-y_{k}\right|: a \in \mathbb{Z}^{3}\right\}$ and set $$v_{k}:=a_{k}* u_{k}=u_{k}\left(\cdot+a_{k}\right).$$
Using the invariance of $E$ and $E^{\pm}$ under the action of $\mathbb{Z}^{3}$, we see that $v_{k} \in E$ and
\begin{equation*}
	\left\|v_{k} \right\|_{L^{2}\left(B_{r+\sqrt{3} / 2}(0)\right)} \geq \frac{\eta}{2}.
\end{equation*}
Moreover, $\left\|v_{k}\right\|=\left\|u_{k}\right\|$, hence $\left\|v_{k}\right\|$ is bounded.
Up to a subsequence (which we continue to denote by $\left\{v_{k}\right\}$), we have $v_{k} \rightharpoonup u$ in $E$ and $v_{k} \rightarrow u$ in $L_{\text {loc }}^{t}\left(\mathbb{R}^{3}, \mathbb{C}^4\right)$, for any $t \in\left[2,3\right)$. It is clear that $\left\|u^{+}\right\|_{L^{2}\left(B_{r+\sqrt{3} / 2}(0)\right)}\geq \frac{\eta}{2}$, which implies $u \neq 0$.
Then by the topological properties of $\Phi_\omega$, we have $u$ is a nontrivial critical point of $\Phi_\omega$.

\medskip
Next, we estimate the regularity of critical points of $\Phi_\omega$. We have obtained that the nontrivial critical point $u$ is in $H^{1/2}(\mathbb{R}^3,\mathbb{C}^4)$.
By using a standard bootstrap argument, we can show higher regularity of these solutions. We sketch the idea of the proof as follows:
\medskip

Step 1. Initial Data: $u$ belongs to $H^{1/2}(\mathbb{R}^3,\mathbb{C}^4)$.

Step 2. Improve $L^p$-integrability of $u$ by iteration as follows.
\[ \xymatrix{
	u\in L^q(\mathbb{R}^3,\mathbb{C}^4) \ar@{=>}[r]  & u \in W^{1,q'}(\mathbb{R}^3,\mathbb{C}^4)  \ar@{=>}[d]
\\
	& u\in L^{q''}(\mathbb{R}^3,\mathbb{C}^4)\ar@{-->}[ul]^{q''>q}                       }\]

Step 3. Obtain that  $u\in W^{1,p}$ for any $p\geq 2$.

\medskip
Assume $u$ is a nontrivial critical point of $\Phi_\omega$, then $u$ is a weak solution of \eqref{eq1.1}. We may write $u=(D-\omega)^{-1} F_u(u)$ as  $0\notin \sigma (D-\omega)$.
By
\[\int_{\mathbb{R}^3} |F_u(u)|^tdx\leq C\|u\|_{L^{(p-1)t}}^{(p-1)t},\]
we have $u\in W^{1,t}(\mathbb{R}^3,\mathbb{C}^4)$ for $t\in [2/(p-1), 3/(p-1)]$. Therefore,
$$u\in L^{n},\quad  n=  \frac{3t}{3-t} .$$
If we have obtained $u\in L^d$ for some $d\geq 2$, then $u\in L^{d'}$ with $d'$ satisfies $1/d'=1/n-1/3$. Since $d'>d$ for $d>2$. Consequently, a standard bootstrap argument shows that
\[u\in W^{1,t}(\mathbb{R}^3,\mathbb{C}^4),\quad \forall t\geq 2.\]
Moreover, $u\in \mathcal{C}^{0,\gamma}$ for some $\gamma \in (0,1)$ by Morrey's inequality.
  $\hfill\Box$
\medskip

\section{Existence results of (NDE) with spectrum zero}\label{sec5}

In this section, we still denote $$\Phi_{-m}(u)=\frac{1}{2}\left( \|u^+\|^2-\|u^-\|^2\right)+\frac{m}{2}\left\|u\right\|_{L^2}^2-\Psi (u) ,\quad \Psi (u)=\int_{\mathbb{R}^3} F( u)dx$$

It is difficult to deal with the functional $\Phi_{-m}$ since we cannot show the $L^2$-boundedness of $(PS)_c$-sequence. Therefore, we introduce the following functional $$\tilde{\Phi}_\lambda (u)=\frac{1}{2}\left( \|u^+\|^2-\|u^-\|^2\right)+\frac{\lambda m}{2}\left\|u\right\|_{L^2}^2-\Psi(u),$$
where $\lambda\in (0,1)$. By the previous discussion, we have for any $\lambda \in (-1,1)$, the functional $\tilde{\Phi}_\lambda$ possesses a critical point $u_\lambda \in E$ with value $c_\lambda$.
Moreover, $c_\lambda$ is formed by
\[c_\lambda=\inf_{w\in E^+} \sup_{u\in E_w} \tilde{\Phi}_\lambda(u),\quad E_w=E^-\oplus \text{span}\ \{w\}.\]
 The key point of the following Lemma is the independence of $\lambda$; this enables us to obtain the uniformly boundedness of the critical value of $\tilde{\Phi}_\lambda$.  
\begin{lemma}Under the assumptions of Theorem \ref{mainthm}, the following conclusions hold:
\begin{itemize}
\item[(i)] There exists $\rho$, $r^*>0$ (independent of $\lambda$), such that $\kappa:=\inf \tilde{\Phi}_\lambda (\partial B_\rho\cap E^+)\geq r^*>0$, for any $\lambda\in (0,1)$.
\item[(ii)] For some $e\in E^+$, with $\|e\|=1$, there exists $R$, $\sigma>0$ (independent of $\lambda$), such that
\[\sup\tilde{\Phi}_\lambda (E_e)<\sigma,\quad \sup\tilde{\Phi}_\lambda (E_e\setminus B_R)\leq 0,\]
where $B_R=\{u\in E_e: \|u\|\leq R\}$.
\end{itemize}
\end{lemma}
\begin{proof}
\begin{itemize}
\item[(i)] Similar to Lemma \ref{geo}, for any $u\in E^+$, by $\lambda >0$, we have
\[\tilde{\Phi}_\lambda (u) \geq \frac{1}{2}\|u\|^2+\frac{\lambda m}{2}\|u\|_{L^2}^2-C\|u\|^p\geq  \frac{1}{2}\|u\|^2 -C\|u\|^p.\]
  Therefore, we can choose $\rho$ small, such that $\kappa\geq r^*>0$, for some $r^*>0$ uniformly with respect to $\lambda$.
\item[(ii)] Let $u=u^-+te\in E_e$, by $\lambda <1$, 
we have 
\begin{align*}
\tilde{\Phi}_\lambda (u)&=\frac{t^2}{2}-\frac{1}{2}\|u^-\|^2+\frac{\lambda m}{2}\|u^-+te\|^2_{L^2}-\Psi(u)\\
&\leq \left(\frac{1}{2}+\frac{m}{2}\|e\|_{L^2}^2\right) t^2-\frac{1}{2}\left(\|u^-\|^2-\lambda m\|u^-\|_{L^2}^2\right)-a_1t^p\|e\|_{L^p}^p\\
&\leq C(t^2-t^p).
\end{align*}
Here, we choose $e\in E^+$ such that $\|e\|_{L^p}\neq 0$. Therefore, there exists $\sigma>0$ (independent of $\lambda$), such that
$\sup\tilde{\Phi}_\lambda (E_e)<\sigma.$ Based on the previous argument, it is clear that we can choose $R$ large, such that
\[\sup\tilde{\Phi}_\lambda (E_e\setminus B_R)\leq 0.\]
\end{itemize}
\end{proof}

It is worth mentioning that if $\lambda\rightarrow -1$, the proof of $(i)$ fails as we may not find a control of the term $\frac{\lambda m}{2}\|u\|_{L^2}^2$.
Based on these two lemmas, we have
\begin{corollary}
There exist $\kappa$, $\sigma >0$ (independent of $\lambda$), such that
\[\kappa\leq c_\lambda \leq \sigma,\quad \forall \lambda \in (0,1).\]
\end{corollary}

The boundedness of $c_\lambda$ plays a crucial role in the proof concerning the Spectrum Zero Problem. It is also noteworthy to consider the continuity of $c_\lambda$ with respect to the parameter $\lambda$. This continuity appears to depend on the reduction properties of $\tilde{\Phi}_\lambda$. We hypothesize that, given certain conditions related to the nonlinear characteristics of the system, the least energy level  $c_\lambda$ exhibits continuity with respect to the parameter  $\lambda$ in the open interval $(-1,1)$.

Next, we choose $\lambda_n =1-\frac{1}{n}$, and set $\tilde{\Phi}_n=\tilde{\Phi}_{\lambda_n}$, $c_n=c_{\lambda_n}$. Moreover, we assume $c_n$ is achieved by $u_n=u_{\lambda_n}\in E$.
Let $c$ be a limit point of $\{c_n\}$. Up to a subsequence, we show $\{u_n\}$ forms a $(PS)_c$-sequence of $\Phi_{-m}\in \mathcal{C}^1(E,\mathbb{R})$.
Before that, we show the uniform boundedness of $\{u_n\}$.
Based on the construction of $\{u_n\}$, we have $u_n$ is a critical point of $\tilde{\Phi}_{n}\in \mathcal{C}^1(E,\mathbb{R})$ with value $c_n$. That is,
\[\tilde{\Phi}_n(u_n)=c_n,\quad \tilde{\Phi}_{n}'(u_n)  =0.\]

\begin{proposition}\label{bound}
$\{u_n\}$ is a bounded sequence in $E$.
\end{proposition}
\begin{proof}
It is clear that
\begin{align*}
\sigma\geq c_n&=\tilde{\Phi}_n(u_n)-\frac{1}{2}\langle \tilde{\Phi}_n'(u_n),u_n\rangle\\
&=\frac{1}{2}\langle \tilde{\Phi}'(u_n),u_n\rangle-\Psi(u_n)\\
&\geq \frac{b_1}{2}  \|u_n^+ \|_{L^p}^p ,
\end{align*}
which implies $\|u_n^+\|_{L^p}<C$. Due to
\[0=\langle \tilde{\Phi}'_n(u_n),u_n^-\rangle=-\|u_n^-\|^2+\lambda_n m\|u_n^-\|_{L^2}^2-\langle \Psi'(u_n),u_n^-\rangle,\]
we have
\[0\leq \|u_n^-\|^2-\lambda_nm\|u_n^-\|_{L^2}^2=-\langle \Psi' (u_n),u_n^-\rangle\leq d_2\left(\|u_n^+\|_{L^p}^p -\|u_n^-\|_{L^2}^2\right).\]
Thus, $\left\{\|u_n^-\|_{L^2}\right\}$ is bounded.
Then, by Lemma \ref{lem2.4} (3)
\[\langle\Psi'(u_n),u_n^+\rangle\leq c_1\|u_n^+\|_{L^p}^p+d_1\|u_n^- \|_{L^2}^2 ,\]
is bounded. Then by
\[0=\langle \tilde{\Phi}'_n(u_n),u_n^+\rangle=\|u_n^+\|^2+\lambda_n m\|u_n^+\|_{L^2}^2-\langle \Psi'(u_n),u_n^+\rangle,\]
we conclude $\|u_n^+\|<C$ and $\|u_n^+\|_{L^2}<C$.
Then we have $\|u_n\|_{L^2}<C$. Due to $$\|u_n^-\|^2\leq m\|u_n^-\|_{L^2}^2+C\left(\|u_n^+\|_{L^p}^p+\|u_n^-\|_{L^p}^p-\|u_n^-\|_{L^2}^2\right),$$
we conclude that  $\|u_n^-\|<C$.  This ends the proof.
\end{proof}

\medskip
Recall that $u_n$ is a critical point of $\tilde{\Phi}_n$, and $\lim\limits_{n\rightarrow \infty} \lambda_n=1$. Comparing the energy functionals $\Phi_{-m}$ and $\tilde{\Phi}_n$, we have
\begin{lemma}
$\{u_n\}$ is a $(PS)_c$-sequence of $\Phi_{-m}\in \mathcal{C}^1(E,\mathbb{R} )$.
\end{lemma}
\begin{proof}
For any $\varphi\in \mathcal{C}_c^\infty(\mathbb{R}^3,\mathbb{C}^4)$, we have
\begin{align*}
\langle \Phi_{-m}'(u_n),\varphi\rangle &=\Re(Du_n,\varphi)_{L^2}+m\Re(u_n,\varphi)_{L^2}-\langle\Psi'(u_n),\varphi\rangle\\
&=(1-\lambda_n)m\Re\left(u_n,\varphi\right)_{L^2}\\
&\leq \frac{m}{n}\|u_n\|_{L^2}^2\|\varphi\|_{L^2}^2\\
&\leq \frac{C}{n}\rightarrow 0.
\end{align*}
Moreover, we have
\[\Phi_{-m}(u_n)=\frac{1}{2}\left(\|u_n^+\|^2-\|u_n^-\|^2\right)+\frac{m}{2}\|u_n\|_{L^2}^2-\Psi(u_n)=c_n+\frac{m}{2n}\|u_n\|_{L^2}^2.\]
Thus, $\Phi_{-m}(u_n)\rightarrow c$ as $n\rightarrow \infty$.
\end{proof}

\medskip

\noindent {\bf Proof of Theorem \ref{mainthm} (2.2).}
Assume $\left\{u_{n}\right\}$ is defined as mentioned above. Then it is a $(PS)_{c}$-sequence of $\Phi_{-m}$. By Proposition \ref{bound}, this sequence is bounded in $E$.
Thus, we may assume that  $u_n\rightharpoonup  u $ in $E$. We claim that for $r>0$ arbitrary there exists a sequence $\left\{y_{n}\right\}$ in $\mathbb{R}^{3}$ and $\eta>0$ such that
\begin{equation*}
	\liminf _{n \rightarrow \infty} \int_{B_r\left(y_{n} \right)}\left|u_{n}\right|^{2} d x \geq \eta .
\end{equation*}
Indeed, if not then by Lemma \ref{Lem3.3}, $u_{n}\rightarrow 0$ in $L^{t}\left(\mathbb{R}^{3}, \mathbb{C}^4\right)$, for any $t \in\left(2,3\right)$.
Moreover, we have
\begin{align*}
\langle \Phi_{-m}'(u_n), u_n^+-u_n^-\rangle =\|u_n \|^2 +m\|u_n^+\|_{L^2}^2-m\|u_n^-\|_{L^2}^2-\langle \Psi'(u_n),u_n^+-u_n^-\rangle,
\end{align*}
and
\[\langle \Psi'(u_n),u_n^+-u_n^-\rangle \leq C\left(\|u_n^+\|_{L^p}^p-\|u_n^-\|_{L^2}^2\right)\leq C\|u_n^+\|_{L^p}^p\rightarrow 0.\]
Therefore, we have $\|u_n^+\|\rightarrow 0$, and $\|u_n^+\|_{L^2}\rightarrow 0$, as $n\rightarrow \infty$. Since
\begin{align*}
\Phi_{-m}(u_n)&=\frac{1}{2}\left(\|u_n^+\|^2-\|u_n^-\|^2\right)+\frac{m}{2}\|u_n\|_{L^2}^2-\Psi(u_n)\\
&=\frac{1}{2}\left(\|u_n^+\|^2+m\|u_n^+\|_{L^2}^2\right)-\frac{1}{2}\left(\|u_n^-\|^2-m\|u_n^-\|_{L^2}^2\right)-\Psi(u_n)\\
&\leq \frac{1}{2}\left(\|u_n^+\|^2+m\|u_n^+\|_{L^2}^2\right).
\end{align*}
This yields $c=\lim\limits_{n \rightarrow \infty} \Phi_{-m}\left(u_{n}\right) \leq 0$, a contradiction.

\medskip

Now we choose $a_{n} \in \mathbb{Z}^{3}$ such that $\left|a_{n}-y_{n}\right|=\min \left\{\left|a-y_{n}\right|: a \in \mathbb{Z}^{3}\right\}$ and set $v_{n}:=a_{n}* u_{n}=u_{n}\left(\cdot+a_{n}\right)$.
Using the invariance of $E$ and $E^{\pm}$ under the action of $\mathbb{Z}^{3}$ we see that $v_{n} \in E$ and
\begin{equation*}
	\left\|v_{n} \right\|_{L^{2}\left(B_{r+\sqrt{3} / 2}(0)\right)} \geq \frac{\eta}{2}.
\end{equation*}
Moreover, $\left\|v_{n}\right\|=\left\|u_{n}\right\|$, hence $\left\|v_{j}\right\|$ is bounded.
Up to a subsequence (which we continue to denote by $\left\{v_{n}\right\}$), we have $v_{n} \rightharpoonup u$ in $E$ and $v_{n} \rightarrow u$ in $L_{\text {loc }}^{t}\left(\mathbb{R}^{3}, \mathbb{C}^4\right)$, for any $t \in\left[2,3\right)$. Due to the claim above, we have $\left\|u \right\|_{L^{2}\left(B_{r+\sqrt{3} / 2}(0)\right)}\geq \frac{\eta}{2}$, which implies $u \neq 0$.
Then by the topological properties of $\Phi_{-m}$, we have $u\in E$ is a nontrivial critical point of $\Phi_{-m}$, which means \eqref{eq1.1} possesses at least one nontrivial solution in $E$.
Similar to the proof of Theorem \ref{mainthm} (2.1), we have $u\in H^1(\mathbb{R}^3,\mathbb{C}^4)$. 
  $\hfill\Box$
\medskip

\section{Proof of the Bifurcation Theorem}\label{sec6}

In this section, we aim to show the proof of Theorem \ref{mainthm2}. First, we recall that 
$$
\Phi_\omega(u)=\frac{1}{2}\left(\left\|u^{+}\right\|^2-\| u^-\|^2\right)-\frac{\omega}{2}\|u\|_{L^2}^2-\Psi(u).
$$
Set
$$
\begin{aligned}
Q_\omega(u) :=\left\|u^{+}\right\|^2-\left\|u^{-}\right\|^2-\omega\|u\|_{L^2}^2  
 =\Re((D-\omega) u, u)_{L^2} .
\end{aligned}
$$
Thus, we have 
$$\Phi_\omega(u)=\frac{1}{2} Q_\omega(u)-\Psi(u).$$
By \cite[Lemma 2.2]{DingYu23}, there exists a nontrivial periodic solution $\tilde{u}_m \in \mathcal{C}^{\infty}\left(\mathbb{R}^3, \mathbb{C}^4\right)$ of the following linear equation $$-i \alpha \cdot \nabla u+m \beta u=m u.$$
It is clear that $\widetilde{u}_m$ does not belong to $L^2\left(\mathbb{R}^3, \mathbb{C}^4\right)$, thus we use a truncation argument to get the desired sequence.
Set
$$
z_k(x):=k^{-\frac{3}{2}} \eta\left(\frac{x}{k}\right) \widetilde{u}_m(x),
$$
where $\eta \in \mathcal{C}_c^{\infty}\left(\mathbb{R}^3,[0,1]\right),$ $\eta(x)=1$, when $|x| \leq 1$; $\eta(x)=0$, when $|x| \geq 2$. Then we have
\begin{lemma}
\begin{itemize}
    \item[(1)] $z_k \in E$, and $\left\|z_k\right\|_{L^2}^2 \rightarrow M\left(\left|\tilde{u}_m\right|^2\right) \int_{\mathbb{R}^3} \eta^2(x) d x$, as $k \rightarrow+\infty$, where $$M(f):=\lim_{T \rightarrow \infty} \frac{1}{T^3} \int_0^T \int_0^T \int_0^T f(x) d x.$$
    \item[(2)]  $\displaystyle Q_m\left(z_k\right)=-i k^{-4} \int_{\mathbb{R}^3} \sum_{j=1}^3 \frac{\partial \eta}{\partial x_j}\left(\frac{x}{k}\right) \alpha_j \tilde{u}_m(x) \cdot \eta\left(\frac{x}{k}\right)   \tilde{u}_m(x) d x$. Moreover, $$Q_m\left(z_k\right)=O\left(\frac{1}{k}\right),\quad \left\|D z_k-m z_k\right\|_{L^2}=O\left(\frac{1}{k}\right).$$
\end{itemize}
\end{lemma}

The details of the above Lemma can be found in the proof of Lemma 2.5 in \cite{DingYu23}.
Based on the spectral decomposition, we set
$z_k^{+}=P^{+} z_k$, choose $k=\frac{1}{m-\omega}$, and set $$\tilde{z}_\omega(x):=z_{\frac{1}{m-\omega}}(x).$$
Then
\begin{align}\label{eq6.1}
Q_\omega\left(\tilde{z}_\omega^{+}\right) =Q_m\left(\tilde{z}_\omega^{+}\right)+(m-\omega)  \|\tilde{z}_\omega \|_{L^2}^2  =O(m-\omega), 
\end{align}
as $\omega \rightarrow m-$.
\begin{lemma}\label{lem6.2}
$\left\|\tilde{z}_\omega^{+}\right\|_{L^p}^p=O\left((m-\omega)^{\frac{3(p-2)}{2}}\right)$, as $\omega \rightarrow m-$.
\end{lemma}
\begin{proof}
It is clear that
$$
\begin{aligned}
\left\|\tilde{z}_\omega^{+}\right\|_{L^p}^p & =\int_{\mathbb{R}^3}(m-\omega)^{\frac{3}{2} p} \eta^p((m-\omega) x)\left|\tilde{u}_m^{+}(x)\right|^p d x \\
& =(m-\omega)^{\frac{3(p-2)}{2}} \int_{\mathbb{R}^3} \eta^p(x) \cdot\left|\tilde{u}_m^{+}\left(\frac{x}{m-\omega}\right)\right|^p d x.
\end{aligned}
$$
Note that $\tilde{u}_m^{+}$ is also a periodic function. By a version of Riemann-Lebesgue Lemma, we have
$$
\int_{\mathbb{R}^3} \eta^p(x)\left|\tilde{u}_m^{+}\left(\frac{x}{m-\omega}\right)^p\right| d x \rightarrow M\left(\left|\tilde{u}_m^{+}\right|^p\right) \cdot \int_{\mathbb{R}^3} \eta^p(x) d x \text {. }
$$
Therefore, $\left\|\tilde{z}_\omega^{+}\right\|_{L^p}^p=O\left((m-\omega)^{\frac{3(p-2)}{2}}\right)$.
\end{proof}
For $v \in E^-$, by Lemma \ref{lem2.4}, we have
$$
\begin{aligned}
\Psi\left(v+s\tilde{z}_\omega^{+}\right)  \geq a_1\left(s^p\left\|\tilde{z}_{\omega}^{+}\right\|_{L^p}^p+\|v\|_{L^2}^2\right)  \geq  a_1 s^p\left\|\tilde{z}_{\omega}^{+}\right\|_{L^p}^p .
\end{aligned}
$$
Thus,
$$
\begin{aligned}
\Phi_\omega\left(v+s \tilde{z}_\omega^{+}\right) & =\frac{1}{2} Q_\omega\left(v+s \tilde{z}_\omega^{+}\right)-\Psi\left(v+s z_\omega^{+}\right) \\
& \leq \frac{1}{2} Q_\omega(v)+\frac{s^2}{2} Q_\omega\left(\tilde{z}_\omega^{+}\right)-a_1 s^p\left\|\tilde{z}_\omega^{+}\right\|_{L^p}^p \\
& \leq  \frac{s^2}{2} Q_\omega\left(\tilde{z}_\omega^{+}\right)-a_1 s^p\left\|\tilde{z}_\omega^{+}\right\|_{L^p}^p.
\end{aligned}
$$
Here, we need the following facts 
$$
\begin{aligned}
 Q_\omega(v)=-\|v\|^2-\omega\|v\|_{L^2}^2 \leq  0, \quad Q_\omega\left(\tilde{z}_\omega^{+}\right)=\left\|\tilde{z}_\omega^{+}\right\|^2-\omega\left\|\tilde{z}_\omega^{+}\right\|_{L^2}^2>0 .
\end{aligned}
$$
Therefore, by \eqref{eq6.1} and Lemma \ref{lem6.2}, we have

\begin{align*}
    \sup_{s>0}\left( \frac{s^2}{2} Q_\omega\left(\tilde{z}_\omega^+\right)-a_1s^p\|\tilde{z}_\omega^+\|_{L^p}^p\right) &\leq \frac{1}{a_1^{\frac{2}{p-2}}}\left(\frac{1}{2p^{\frac{2}{p-2}}}-\frac{1}{p^{\frac{p}{p-2}}}\right) \left(\frac{Q_\omega(\tilde{z}_\omega^+)}{\|\tilde{z}_\omega^+\|_{L^p}^2}\right)^{\frac{p}{p-2}}\\
    &=O\left((m-\omega)^{\frac{6-2p}{p-2}}\right),
\end{align*} 
as $\omega \rightarrow m-$.
Thus,
$$
\begin{aligned}
c_\omega: & =\inf _{w \in E^{+}} \sup _{u \in E_\omega} \Phi_\omega(u) \\
& \leq  \sup _{\substack{v\in E^{-} \\
s \geq 0}} \Phi_\omega\left(v+s \tilde{z}_\omega^{+}\right) \\
& \leq C\left((m-\omega)^{\frac{6-2 p}{p-2}}\right) .
\end{aligned}
$$
{\bf Proof of Theorem \ref{mainthm2}.}
 For $\omega \in(-m, m)$, we have shown that $c_\omega$ is attained, let $u_\omega$ be the critical point of $\Phi_\omega$ with $\Phi_\omega\left(u_\omega\right)=c_\omega$.
Thus,
$$
\begin{aligned}
c_\omega & =\Phi_\omega\left(u_\omega\right)-\frac{1}{2}\left\langle\Phi_\omega^{\prime}\left(u_\omega\right), u_\omega\right\rangle \\
& =\frac{1}{2}\left\langle\Psi^{\prime}\left(u_\omega\right), u_\omega\right\rangle-\Psi\left(u_\omega\right) \\
& \geq \frac{b_1}{2}\left\|u_\omega^{+}\right\|_{L^p}^p .
\end{aligned}
$$
This implies that $$\left\|u_\omega^{+}\right\|_{L^p} \leq  C  (m-\omega)^{\frac{6-2 p}{p(p-2)}},$$ as $\omega \rightarrow m-$. Since $$\left\langle\Phi_\omega^{\prime}\left(u_\omega\right), u_\omega^{+}-u_{\omega}^{-}\right\rangle=\left\|u_\omega\right\|^2-\omega\left(|| u_\omega^{+}\left\|_{L^2}^2-\right\| u_\omega^{-} \|_{L^2}^2\right)-\left\langle\Psi^{\prime}\left(u_\omega\right), u_{\omega}^{+}-u_\omega^{-}\right\rangle,$$ 
combined with 
$$
\left\|u_\omega\right\|^2=\omega\left(\|u_\omega^{+}\left\|_{L^2}^2-\right\| u_\omega^{-} \|_{L^2}^2\right)+\left\langle\Psi^{\prime}\left(u_\omega\right), u_\omega^{+}-u_\omega^{-}\right\rangle,
$$
we get
$$
\begin{aligned}
\frac{m-\omega}{m}\left\|u_\omega\right\|^2 &\leq \| u_\omega \|^2-\omega\left(\|u_\omega^{+}\|_{L^2} ^2-\left\|u_\omega^{-}\right\|_{L^2}^2\right) \\
& =\left\langle\Psi^{\prime}\left(u_\omega\right), u_\omega^{+}-u_\omega^{-}\right\rangle \\
& \leq C\left(\left\|u_\omega^{+}\right\|_{L^p}^{p}-\left\|u_\omega^{-}\right\|_{L^2}^2\right) \\
& \leq C\left\|u_\omega^{+}\right\|_{L^p}^{p}  \\
& \leq C  (m-\omega)^{\frac{6-2 p}{p-2}}.\\
&
\end{aligned}
$$
Therefore, we have  $$\left\|u_\omega\right\| \leq C  (m-\omega)^{\frac{8-3 p}{2(p-2)}} \rightarrow 0,$$ as $\omega \rightarrow m-$ when $2<p<8/3$.

\small {\bf \small  Declarations of interest}: none.

{\bf \small Data availability statement:} There are no new data associated with this article.


\bibliographystyle{plain}
\bibliography{Dirac}

 \newpage
\noindent {Qi Guo\\
School of Mathematics,\\
Renmin University of China, Beijing, 100872, P.R. China\\
e-mail: qguo@ruc.edu.cn}
\medskip
\\
\noindent{Yuanyuan Ke\\
School of Mathematics,\\
Renmin University of China, Beijing, 100872, P.R. China\\
e-mail: keyy@ruc.edu.cn}
\medskip
\\
\noindent{Bernhard Ruf\\
Accademia di Scienze e Lettere,\\
Istituto Lombardo, Milano, 20133, Italy\\
e-mail: bruff001@gmail.com}

\end{document}